\documentclass[final]{amsart}
\usepackage{amssymb,amsmath, color}

\usepackage{amssymb}

\usepackage[notref,notcite]{showkeys}

\newtheorem{Theorem}{Theorem} 
\newtheorem{lemma}[Theorem]{Lemma}

\newtheorem{fact}{Fact}
\theoremstyle{definition}
\newtheorem{remark}{Remark}

\newcommand{\R}{{\mathbb R}}
\newcommand{\N}{{\mathbb N}}
\DeclareMathOperator{\argmin}{argmin}
\newcommand{\interior}{{\rm int}\kern 0.06em}

\usepackage{ulem}

\begin{document}

\title{The rate of convergence of Nesterov's accelerated forward-backward method is actually faster than $1/k^{-2}$}

\author{Hedy Attouch}
\address{Institut de Math\'ematiques et Mod\'elisation de Montpellier, UMR 5149 CNRS, Universit\'e Montpellier 2, place Eug\`ene Bataillon,
34095 Montpellier cedex 5, France}
\email{hedy.attouch@univ-montp2.fr}

\author{Juan Peypouquet}
\address{Departamento de Matem\'atica, Universidad T\'ecnica Federico Santa Mar\'\i a, Av Espa\~na 1680, Valpara\'\i so, Chile}
\email{juan.peypouquet@usm.cl}



\keywords{Convex optimization, fast convergent methods, Nesterov method}

\thanks{Effort sponsored by the Air Force Office of Scientific Research, Air Force Material Command, USAF, under grant number FA9550-14-1-0056. Also supported by Fondecyt Grant 1140829, Conicyt Anillo ACT-1106, ECOS-Conicyt Project C13E03, Millenium Nucleus ICM/FIC RC130003, Conicyt Project MATHAMSUD 15MATH-02, Conicyt Redes 140183, and Basal Project CMM Universidad de Chile. Part of this research was carried out while the authors were visiting Hangzhou Dianzi University by invitation of Professor Hong-Kun Xu.}

\begin{abstract}
The {\it forward-backward algorithm} is a powerful tool for solving optimization problems with a {\it additively separable} and {\it smooth} + {\it nonsmooth} structure. In the convex setting, a simple but ingenious acceleration scheme developed by Nesterov has been proved useful to improve the theoretical rate of convergence for the function values from the standard $\mathcal O(k^{-1})$ down to $\mathcal O(k^{-2})$. In this short paper, we prove that the rate of convergence of a slight variant of Nesterov's accelerated forward-backward method, which produces {\it convergent} sequences, is actually $o(k^{-2})$, rather than $\mathcal O(k^{-2})$. Our arguments rely on the connection between this algorithm and a second-order differential inclusion with vanishing damping.\\

Final version published at SIOPT.
\end{abstract}

\maketitle

\section*{Introduction}

Let $\mathcal H$ be a real Hilbert space endowed with the scalar product $\langle \cdot,\cdot\rangle$ and norm $\|\cdot\|$, and consider the problem
\begin{equation}\label{algo1}
\min \left\lbrace  \Psi (x) + \Phi (x): \ x \in \mathcal H   \right\rbrace   
\end{equation}
where $\Psi: \mathcal H \to  \mathbb R \cup \lbrace   + \infty  \rbrace $ is a proper lower-semicontinuous convex function, and $\Phi: \mathcal H \to  \mathbb R$ is a continuously differentiable convex function, whose gradient is Lipschitz continuous. 

Based on the {\it gradient projection} algorithm of \cite{Gol} and \cite{LevPol}, the {\it forward-backward} method was proposed in \cite{LioMer}, and \cite{Pas} to overcome the inherent difficulties of minimizing the nonsmooth sum of two functions, as in \eqref{algo1}, while exploiting its {\it additively separable} and {\it smooth} + {\it nonsmooth} structure. It gained popularity in image processing following \cite{DauDefDeM} and \cite{ComWaj}: when $\Psi$ is the $\ell^1$ norm in $\R^N$ and $\Phi$ is quadratic, this gives the {\it Iterative Shrinkage-Thesholding Algorithm} (ISTA). Some time later, a decisive improvement came with \cite{BT}, where ISTA was successfully combined with Nesterov's acceleration scheme \cite{Nest1} producing the {\it Fast Iterative Shrinkage-Thesholding Algorithm} (FISTA). For general $\Phi$ and $\Psi$, and after some simplification, the {\it Accelerated Forward-Backward} method can be written as
\begin{equation}\label{algo7b}
\left\{
\begin{array}{rcl}
y_k&= &  x_{k} + \frac{k -1}{k  + \alpha -1} ( x_{k}  - x_{k-1}) \\
\rule{0pt}{20pt}
x_{k+1}& = &\mbox{prox}_{s \Psi} \left( y_k- s (\nabla \Phi (y_k))\right),
\end{array}\right.
\end{equation}
where $\alpha>0$ and $s>0$. This algorithm is also in close connection with the proximal-based inertial algorithms \cite{AA1}, \cite{MO} and \cite{LP}. The choice $\alpha=3$ is current common practice. The remarkable property of this algorithm is that, despite its simplicity and computational efficiency $-$equivalent to that of the classical forward-backward method$-$, it guarantees a rate of convergence of $\mathcal O(k^{-2})$, where $k$ is the number of iterations, for the minimization of the function values, instead of the classical $\mathcal O(k^{-1})$ that is obtained for the unaccelerated counterpart. However, while sequences generated by the classical forward backward method are convergent, the convergence of the sequence $(x_k)$ generated by \eqref{algo7b} to a minimizer of $\Phi+\Psi$ puzzled researchers for over two decades. This question was recently settled in \cite{CD} and \cite{ACPR} independently, and using different arguments. In \cite{CD}, the authors use a {\it descent} inequality satisfied by forward-backward iterations. A perspicuous abstract presentation of this idea is given in \cite[Section 2.2]{Chamb_Pock}. In turn, the proof given in \cite{ACPR} relies on the connection between \eqref{algo7b} and the differential inclusion
\begin{equation}\label{algo2}
\ddot{x}(t) + \frac{\alpha}{t} \dot{x}(t)  + \partial \Psi (x(t)) + \nabla \Phi (x(t)) \ni 0.
\end{equation}
Indeed, as pointed out in \cite{SBC,ACPR}, algorithm \eqref{algo7b} can be seen as an appropriate finite-difference discretization of \eqref{algo2}. In \cite{SBC}, the authors studied 
\begin{equation}\label{E:diff_eq}
\ddot{x}(t) + \frac{\alpha}{t} \dot{x}(t)  + \nabla \Theta (x(t)) = 0.
\end{equation}
and proved that $$\Theta(x(t))-\min\Theta=\mathcal O(t^{-2})$$ when $\alpha\ge 3$. Convergence of the trajectories was obtained in \cite{ACPR} for $\alpha>3$. The study of the long-term behavior of the trajectories satisfying this evolution equation has given important insight into Nesterov's acceleration method and its variants, and the present work is inspired in this relationship. If $\alpha>3$, we actually have $$\Theta(x(t))-\min\Theta=o(t^{-2}).$$
Although it can be derived from the arguments in \cite{ACPR}, it was May \cite{May} who first pointed out this fact, giving a different proof. This is another justification for the interest of taking $\alpha >3$ instead of $\alpha=3$.

The purpose of this paper is to show that sequences generated by Nesterov's accelerated version of the forward-backward method approximate the optimal value of the problem with a rate that is strictly faster than $\mathcal O(k^{-2})$. More precisely, we prove the following:

\begin{Theorem} \label{T:main}
Let  $\Psi: \mathcal H \to  \mathbb R \cup \lbrace + \infty  \rbrace  $ be proper, lower-semicontinuous and convex, and let $\Phi: \mathcal H \to  \mathbb R$ be convex and continuously differentiable with $L$-Lipschitz continuous gradient. Suppose that $S=\argmin (\Psi + \Phi)\neq\emptyset$, and let
$(x_k)$ be a sequence generated by algorithm \eqref{algo7b} with $\alpha > 3$ and $ 0< s < \frac{1}{L} $. Then, the function values and the velocities satisfy
$$\lim\limits_{k\to\infty}k^2\Big((\Psi + \Phi)(x_k)-  \min(\Psi + \Phi)\Big) = 0\qquad\hbox{and}\qquad\lim\limits_{k\to\infty}k\|x_{k+1}-x_{k}\|=0,$$ 
respectively. In other words, $$(\Psi + \Phi)(x_k)-  \min(\Psi + \Phi) =o(k^{-2})\qquad\hbox{and}\qquad\|x_{k+1}-x_{k}\|=o(k^{-1}).$$
\end{Theorem}

Moreover, we recover some results from \cite[Section 5]{ACPR}, closely connected with the ones in \cite{CD}, with simplified arguments. As shown in \cite[Example 2.13]{ACPR}, there is no $p>2$ such that the order of convergence is $\mathcal O(k^{-p})$ for every $\Phi$ and $\Psi$. In this sense, Theorem \ref{T:main} is optimal.

We close this paper by establishing a tolerance estimation that guarantees that the order of convergence is preserved when the iterations given in \eqref{algo7b} are computed inexactly (see Theorem \ref{T:errors}). Inexact FISTA-like algorithms have also been considered in \cite{SLB,VSBV}.


\section{Main results}

Throughout this section, $\Psi: \mathcal H \to  \mathbb R \cup \lbrace   + \infty  \rbrace  $ is proper, lower-semicontinuous and convex, and $\Phi: \mathcal H \to  \mathbb R$ is convex and continuously differentiable with $L$-Lipschitz continuous gradient. To simplify the notation, we set $\Theta = \Psi + \Phi$. We assume that $S=\argmin (\Psi + \Phi)\neq\emptyset$, and consider a sequence $(x_k)$ generated by algorithm \eqref{algo7b} with $\alpha \geq 3$ and $0<s<\frac{1}{L}$. For standard notation and convex analysis background, see \cite{BC,Pey}.

\subsection{Some important estimations}

We begin by establishing the basic properties of the sequence $(x_k)$. Some results can be found in \cite{CD,ACPR}, for which we provide simplified proofs.\\

Let $x^{*} \in \argmin \Theta$. For each $k\in\N$, set 
\begin{equation}\label{algo9b}
\mathcal E (k):= \frac{2s}{\alpha -1} \left( k + \alpha -2\right)^2    (\Theta (x_k) - \Theta(x^{*})) + (\alpha -1)
 \| z_k -x^{*} \|^2,
\end{equation}
where
\begin{equation}\label{algo9c}
z_k := \frac{k + \alpha -1}{\alpha -1}y_k - \frac{k}{\alpha -1}x_k= x_{k} + \frac{k -1}{\alpha -1} ( x_{k}  - x_{k-1}).
\end{equation}
The key idea is to verify that the sequence $(\mathcal E(k))$ has Lyapunov-type properties. By introducing the operator $G_s: \mathcal H \rightarrow \mathcal H$, defined by
$$
G_{s} (y) = \frac{1}{s}\left( y - \mbox{prox}_{s \Psi} \left( y- s \nabla \Phi(y)  \right) \right)
$$
for each $y \in \mathcal H$, the formula for $x_{k+1}$ in algorithm \eqref{algo7b} can be rewritten as
\begin{equation}\label{E:x_{k+1}}
x_{k+1} = y_k - s G_{s} (y_k).
\end{equation}
The variable $z_k$, defined in \eqref{algo9c}, will play an important role. Simple algebraic manipulations give
\begin{equation}\label{algo14a}
z_{k+1}= \frac{k + \alpha -1}{\alpha -1}   \left(y_k - s G_{s} (y_k) \right)  -\frac{k}{\alpha -1}x_{k}
= z_k -\frac{s}{\alpha -1} \left( k + \alpha -1\right)  G_{s} (y_k).
\end{equation}
The operator $G_{s}$ satisfies
\begin{equation}\label{algo14}
 \Theta(y - sG_{s} (y)) \leq \Theta (x) + \left\langle  G_{s} (y), y-x \right\rangle -\frac{s}{2} \|  G_{s} (y) \|^2 .
\end{equation}
for all $x, y\in \mathcal H$ (see \cite{BT}, \cite{CD}, \cite{PB}, \cite{SBC}), since $s \leq \frac{1}{L}$, and $\nabla \Phi$ is $L$-lipschitz continuous. Let us write successively this formula at $y=y_k$ and $x= x_k$, then at $y=y_k$ and $x= x^{*}$. We obtain 
\begin{equation}\label{E:G_xk_yk}
\Theta(y_k - sG_{s} (y_k)) \leq \Theta(x_k) + \left\langle  G_{s} (y_k), y_k-x_k \right\rangle -\frac{s}{2} \|  G_{s} (y_k) \|^2
\end{equation}
and
\begin{equation} \label{E:G_xk_x*}
\Theta(y_k - sG_{s} (y_k)) \leq \Theta(x^{*}) + \left\langle  G_{s} (y_k), y_k-x^{*} \right\rangle -\frac{s}{2} \|  G_{s} (y_k) \|^2,
\end{equation} 
respectively. Multiplying the first inequality by $\frac{k}{k + \alpha -1}$, and the second one by $\frac{\alpha -1}{k + \alpha -1}$, then adding the two resulting inequalities, and using the fact that $x_{k+1} = y_k - s G_{s} (y_k)$, we obtain
$$
\Theta(x_{k+1}) \leq \frac{k}{k + \alpha -1} \Theta(x_k) + \frac{\alpha -1}{k + \alpha -1}\Theta(x^{*})-\frac{s}{2} \|G_{s} (y_k) \|^2 + \left\langle  G_{s} (y_k), \frac{k}{k + \alpha -1}(y_k-x_k ) + \frac{\alpha -1}{k + \alpha -1} (y_k-x^{*})\right\rangle .
$$
Since
$$
\frac{k}{k + \alpha -1}(y_k-x_k ) + \frac{\alpha -1}{k + \alpha -1} (y_k-x^{*})	=\frac{\alpha -1}{k + \alpha -1}(z_k-x^*),
$$
we obtain
\begin{equation}\label{algo17}
 \Theta(x_{k+1}) \leq \frac{k}{k + \alpha -1} \Theta(x_k) + \frac{\alpha -1}{k + \alpha -1}\Theta(x^{*}) +
 \frac{\alpha -1}{k + \alpha -1} 
 \left\langle  G_{s} (y_k), z_k -x^{*}\right\rangle -\frac{s}{2} \|  G_{s} (y_k) \|^2 .
\end{equation}
We shall obtain a recursion from  \eqref{algo17}. To this end, observe that \eqref{algo14a} gives
$$
z_{k+1} -x^{*} = z_k  -x^{*} -\frac{s}{\alpha -1} \left( k + \alpha -1\right)  G_{s} (y_k) .
$$   
After developing
$$
\| z_{k+1} -x^{*} \|^2 = \| z_{k} -x^{*} \|^2 
-2\frac{s}{\alpha -1} \left( k + \alpha -1\right) 
\left\langle   z_{k} -x^{*}, G_{s} (y_k)  \right\rangle  + \frac{s^2}{(\alpha -1)^2} \left( k + \alpha -1\right)^2  \|  G_{s} (y_k)  \|^2 ,
$$
and multiplying the above expression by 
$\frac{(\alpha -1)^2} {2s\left( k + \alpha -1\right)^2}$,
we obtain
$$
\frac{(\alpha -1)^2} {2s\left( k + \alpha -1\right)^2}
\left( \| z_{k} -x^{*} \|^2 -\| z_{k+1} -x^{*} \|^2 \right)  = \frac{\alpha -1}{k + \alpha -1} 
 \left\langle  G_{s} (y_k), z_k -x^{*}\right\rangle -\frac{s}{2} \|  G_{s} (y_k) \|^2 .
$$
Replacing this in \eqref{algo17}, we deduce that
\begin{equation*}
 \Theta(x_{k+1}) \leq \frac{k}{k + \alpha -1} \Theta(x_k) + \frac{\alpha -1}{k + \alpha -1}\Theta(x^{*}) +
 \frac{(\alpha -1)^2} {2s\left( k + \alpha -1\right)^2}\left( \| z_{k} -x^{*} \|^2 -\| z_{k+1} -x^{*} \|^2 \right) .
\end{equation*}
Equivalently,
\begin{equation*}
 \Theta(x_{k+1})-\Theta(x^{*}) \leq \frac{k}{k + \alpha -1} \left( \Theta(x_k) -\Theta(x^{*}) \right) + \frac{(\alpha -1)^2} {2s\left( k + \alpha -1\right)^2}\left( \| z_{k} -x^{*} \|^2 -\| z_{k+1} -x^{*} \|^2 \right) .
\end{equation*}
Multiplying by $ \frac{2s}{\alpha -1}\left( k + \alpha -1\right)^2    $, we obtain
\begin{align*}
\frac{2s}{\alpha -1}\left( k + \alpha -1\right)^2    \left( \Theta(x_{k+1})-\Theta (x^{*})\right)  \leq &\frac{2s}{\alpha -1} k\left( k + \alpha -1\right) \left( \Theta (x_k) -\Theta(x^{*}) \right) + (\alpha -1)
\left( \| z_{k} -x^{*} \|^2 -\| z_{k+1} -x^{*} \|^2 \right),
\end{align*}
which implies
$$
\frac{2s}{\alpha -1}\left( k + \alpha -1\right)^2    \left( \Theta(x_{k+1})-\Theta (x^{*})\right) +  2s \frac{\alpha -3}{\alpha -1}k \left( \Theta (x_k) -\Theta(x^{*}) \right)$$ 
$$\leq \frac{2s}{\alpha -1} \left( k + \alpha -2\right)^2 \left( \Theta (x_k) -\Theta(x^{*}) \right) +(\alpha -1)
\left( \| z_{k} -x^{*} \|^2 -\| z_{k+1} -x^{*} \|^2 \right),
$$
in view of
$$
k\left( k + \alpha -1\right) =\left( k + \alpha -2\right)^2 -k(\alpha -3) -(\alpha -2)^2 \leq \left( k + \alpha -2\right)^2 -k(\alpha -3).
$$
In other words,
\begin{equation} \label{E:energy_decrease}
\mathcal E(k+1) +  2s \frac{\alpha -3}{\alpha -1}k \left( \Theta (x_k) -\Theta(x^{*}) \right) \leq  \mathcal E(k).
\end{equation} 

We deduce the following:

\begin{fact}\label{F:energy_decrease}
The sequence $\big(\mathcal E(k)\big)$ is nonincreasing and $\lim\limits_{k\to\infty}\mathcal E(k)$ exists.
\end{fact}

In particular, $\mathcal E(k)\le\mathcal E(0)$ and we have:

\begin{fact}\label{F:Thetak=Ok2}
For each $k\ge 0$, we have $\displaystyle\Theta (x_k) - \Theta(x^{*})\le \frac{(\alpha-1)\mathcal E(0)}{2s( k + \alpha -2)^2}$ and $\displaystyle\| z_{k} -x^{*} \|^2 \le \frac{\mathcal E(0)}{\alpha-1}$.
\end{fact}

From \eqref{E:energy_decrease}, we also obtain:

\begin{fact}\label{F:kThetak_finite}
If $\alpha>3$, then
$\displaystyle\sum_{k=1}^\infty k\Big(\Theta (x_k) - \Theta(x^{*})\Big)\le \frac{(\alpha-1)\mathcal E(1)}{2s(\alpha-3)}$.
\end{fact}

Now, using \eqref{E:G_xk_yk} and recalling that $x_{k+1}=y_k-sG_s(y_k)$ and $y_k - x_{k} = \frac{k -1}{k  + \alpha -1} ( x_{k}  - x_{k-1}) $, we  obtain
\begin{equation} \label{E:mechanical_energy_decrease_pre}
 \Theta(x_{k+1}) + \frac{1}{2s}\| x_{k+1}-x_k \|^2      \leq \Theta(x_k) + \frac{1}{2s}\frac{(k -1)^2}{(k  + \alpha -1)^2}  \|  x_{k}  - x_{k-1} \|^2 .
\end{equation}
Subtract $\Theta(x^*)$ on both sides, and set $\theta_k := \Theta (x_k) -\Theta (x^{*}) $ and $d_k: = \frac{1}{2s}\|  x_{k+1}  - x_{k} \|^2 $. We can write \eqref{E:mechanical_energy_decrease_pre} as
\begin{equation} \label{E:mechanical_energy_decrease}
\theta_{k+1} + d_k \leq \theta_k + \frac{(k -1)^2}{(k  + \alpha -1)^2}d_{k-1}.
\end{equation}

Since $k+\alpha-1\ge k+1$, \eqref{E:mechanical_energy_decrease} implies
$$(k+1)^2d_k-(k -1)^2d_{k-1}\le (k+1)^2(\theta_k-\theta_{k+1}).$$
But then
$$(k+1)^2(\theta_k-\theta_{k+1})=k^2\theta_k-(k+1)^2\theta_{k+1}+(2k+1)\theta_k\le k^2\theta_k-(k+1)^2\theta_{k+1}+3k\theta_k$$
for $k\ge 1$, and so
\begin{eqnarray*}
2kd_k+k^2d_k-(k-1)^2d_{k-1} & \le & (k+1)^2d_k-(k -1)^2d_{k-1}\\
& \le & (k+1)^2(\theta_k-\theta_{k+1})\\
& \le & k^2\theta_k-(k+1)^2\theta_{k+1}+3k\theta_k
\end{eqnarray*}
for $k\ge 1$. Summing for $k=1,\dots,K$, we obtain
$$K^2d_K+2\sum_{k=1}^Kkd_k\le \theta_1+\frac{3(\alpha-1)\mathcal E(1)}{2s(\alpha-3)}$$
in view of Fact \ref{F:kThetak_finite}. In particular, we obtain

\begin{fact}\label{F:k_vk2_finite}
If $\alpha>3$, then
$\displaystyle\sum_{k=1}^\infty k\|  x_{k+1}  - x_{k} \|^2\le \frac{\alpha(3\alpha-5)\mathcal E(1)}{4s(\alpha-1)(\alpha-3)}$.
\end{fact}

\begin{remark}
Observe that the upper bounds given in Facts \ref{F:kThetak_finite} and \ref{F:k_vk2_finite} tend to $\infty$ as $\alpha$ tends to 3.
\end{remark}

\subsection{From $\mathcal O(k^{-2})$ to $o(k^{-2})$}\label{petit-o-discret}

Recall that $\Psi: \mathcal H \to  \mathbb R \cup \lbrace + \infty  \rbrace  $ is proper, lower-semicontinuous and convex, $\Phi: \mathcal H \to  \mathbb R$ is convex and continuously differentiable with $L$-Lipschitz continuous gradient, and $\Theta=\Phi+\Psi$. We suppose that $S=\argmin (\Psi + \Phi)\neq\emptyset$, and let $(x_k)$ be a sequence generated by algorithm \eqref{algo7b} with $\alpha > 3$ and $ 0< s < \frac{1}{L} $. We shall prove that the function values and the velocities satisfy
$$\lim\limits_{k\to\infty}k^2\Big((\Psi + \Phi)(x_k)-  \min(\Psi + \Phi)\Big) = 0\qquad\hbox{and}\qquad\lim\limits_{k\to\infty}k\|x_{k+1}-x_{k}\|=0,$$ 
respectively. In other words, $(\Psi + \Phi)(x_k)-  \min(\Psi + \Phi) =o(k^{-2})$ and $\|x_{k+1}-x_{k}\|=o(k^{-1})$.\\

The following result is new, and will play a central role in the proof of Theorem \ref{T:main}.

\begin{lemma}\label{F:lim_kvk_exists}
If $\alpha>3$, then $\lim\limits_{k\to\infty}\Big[k^2\|  x_{k+1}  - x_{k} \|^2+(k+1)^2\big(\Theta(x_{k+1})-\Theta(x^*)\big)\Big]$ exists.
\end{lemma}

\begin{proof}
Since $k+\alpha-1\ge k$, inequality \eqref{E:mechanical_energy_decrease} gives
$$k^2d_k-(k -1)^2d_{k-1}\le k^2(\theta_k-\theta_{k+1}).$$
But
$$(k+1)^2\theta_{k+1}-k^2\theta_k=k^2(\theta_{k+1}-\theta_k)+(2k+1)\theta_{k+1}\le k^2(\theta_{k+1}-\theta_k)+2(k+1)\theta_{k+1},$$
and so
	\begin{equation} \label{E:dk_thetak}
\Big[k^2d_k+(k+1)^2\theta_{k+1}\Big]-\Big[(k -1)^2d_{k-1}+k^2\theta_k\Big]\le 2(k+1)\theta_{k+1}.
\end{equation} 
The result is obtained by observing that $k^2d_k+(k+1)^2\theta_{k+1}$ is bounded from below and the right-hand side of \eqref{E:dk_thetak} is summable (by Fact \ref{F:kThetak_finite}). 
\end{proof}

We are now in a position to prove Theorem \ref{T:main}. \\

\noindent{\bf Proof of Theorem \ref{T:main}.} From Facts \ref{F:kThetak_finite} and \ref{F:k_vk2_finite}, we deduce that
$$\sum\limits_{k=1}^\infty \frac{1}{k}\Big[k^2\|  x_{k+1}  - x_{k} \|^2+(k+1)^2\big(\Theta(x_{k+1})-\Theta(x^*)\big)\Big]<+\infty.$$
Combining this with Lemma \ref{F:lim_kvk_exists}, we obtain
$$\lim\limits_{k\to\infty}\Big[k^2\|  x_{k+1}  - x_{k} \|^2+(k+1)^2\big(\Theta(x_{k+1})-\Theta(x^*)\big)\Big]=0.$$
Since all the terms are nonnegative, we conclude that both limits are 0, as claimed.\hfill$\blacksquare$

\begin{remark}
Facts \ref{F:kThetak_finite} and \ref{F:k_vk2_finite}, also imply that the function values and the velocities satisfy
$$\liminf\limits_{k\to\infty}k^2\ln(k)\Big((\Psi + \Phi)(x_k)-  \min(\Psi + \Phi)\Big) = 0\qquad\hbox{and}\qquad\liminf\limits_{k\to\infty}k\ln(k)\|x_{k+1}-x_{k}\|=0,$$ 
respectively. Indeed, if $\beta_k$ is any nonnegative sequence such that $\sum\limits_{k=1}^{\infty}\frac{\beta_k}{k}<\infty$ (which holds for $(k^2d_k)$ and $(k^2\theta_k)$), then it cannot be true that $\liminf\limits_{k\to\infty}\beta_k\ln(k)\ge \varepsilon>0$. Otherwise, $\frac{\beta_k}{k}\ge\frac{\varepsilon}{k\ln(k)}$ for all sufficiently large $k$, and the series above would be divergent.
\end{remark}

\subsection{Convergence of the sequence.}

It is possible to prove that the sequences generated by \eqref{algo7b} converge weakly to minimizers of $\Psi+\Phi$ when $\alpha>3$. Although this was already shown in \cite{ACPR}, we provide a proof following the preceding ideas, for completeness.

\begin{Theorem}
Let  $\Psi: \mathcal H \to  \mathbb R \cup \lbrace + \infty  \rbrace  $ be proper, lower-semicontinuous and convex, and let $\Phi: \mathcal H \to  \mathbb R$ be convex and continuously differentiable with $L$-Lipschitz continuous gradient. Suppose that $S=\argmin (\Psi + \Phi)\neq\emptyset$, and let
$(x_k)$ be a sequence generated by algorithm \eqref{algo7b} with $\alpha > 3$ and $ 0< s < \frac{1}{L} $. Then, the sequence $(x_k)$ converges weakly to a point in $S$.
\end{Theorem}

\begin{proof}
Using the definition \eqref{algo9c} of $z_{k}$, we write 
\begin{eqnarray*}
\| z_k -x^{*} \|^2 & = & \left(\frac{k -1} {\alpha -1}\right)^2\| x_{k} - x_{k-1}\|^2 + 2\frac{k-1}{\alpha -1} \left\langle x_{k} -x^{*}, x_{k} - x_{k-1} \right\rangle + \| x_{k} -x^{*}\|^2\\
& = & \left[\left(\frac{k -1} {\alpha -1}\right)^2+\left(\frac{k -1} {\alpha -1}\right)\right]\| x_{k} - x_{k-1}\|^2+\left(\frac{k -1} {\alpha -1}\right)\Big[\|x_k-x^*\|^2-\|x_{k-1}-x^*\|^2\Big]+\| x_{k} -x^{*}\|^2.
\end{eqnarray*}
We shall prove that $\lim\limits_{k\to\infty}\| z_k -x^{*} \|$ exists. By Lemma \ref{F:lim_kvk_exists} (or Theorem \ref{T:main}) and Fact \ref{F:k_vk2_finite}, it suffices to prove that
$$\delta_k:=(k-1)\Big[\|x_k-x^*\|^2-\|x_{k-1}-x^*\|^2\Big]+(\alpha -1)\| x_{k} -x^{*}\|^2$$
has a limit as $k\to\infty$. Clearly, $(\delta_k)$ is bounded, by Facts \ref{F:Thetak=Ok2} and \ref{F:k_vk2_finite}. Write $h_k:=\| x_{k} -x^{*}\|^2$ and notice that
\begin{eqnarray} \label{E:delta_k_k-1}
\delta_{k+1} -  \delta_{k}  & = & (\alpha -1)( h_{k+1} -h_k)+ k(h_{k+1} -h_{k} ) -(k-1)(h_{k} -h_{k-1} )\nonumber \\
& = & (k+ \alpha -1)( h_{k+1} -h_k)  -(k-1)(h_{k} -h_{k-1} ).
\end{eqnarray}

On the other hand, from \eqref{E:G_xk_x*}, we obtain
$$\Theta(x_{k+1})-\Theta(x^*)\le \langle G_s(y_k),y_k-x^*\rangle -\frac{s}{2}\|G_s(y_k)\|^2.$$
Since $x_{k+1}=y_k-sG_s(y_k)$, we have
\begin{eqnarray*}
0 & \le & 2\langle y_k-x_{k+1},y_k-x^*\rangle -\|y_k-x_{k+1}\|^2\\
& = & \|y_k-x_{k+1}\|^2+\|y_k-x^*\|^2-\|x_{k+1}-x^*\|^2-\|y_k-x_{k+1}\|^2,
\end{eqnarray*}
and so
\begin{eqnarray*}
\|x_{k+1}-x^*\|^2 & \le & \|y_k-x^*\|^2\\
& = & \left\|x_{k}-x^* + \frac{k -1}{k  + \alpha -1} ( x_{k}  - x_{k-1})\right\|^2\\
& = & \|x_{k}-x^*\|^2+\left(\frac{k-1}{k+\alpha-1}\right)^2\|x_{k}  - x_{k-1}\|^2+2\frac{k-1}{k+\alpha-1}\langle x_{k}-x^*,x_{k}- x_{k-1}\rangle\\
& = & \|x_{k}-x^*\|^2+\left[\left(\frac{k-1}{k+\alpha-1}\right)^2+\frac{k-1}{k+\alpha-1}\right]\|x_{k}  - x_{k-1}\|^2 +\frac{k-1}{k+\alpha-1}\Big[\|x_{k+1}-x^*\|^2-\|x_k-x^*\|^2\Big]\\
& \le & \|x_{k}-x^*\|^2+2\|x_k-x_{k-1}\|^2+\frac{k-1}{k+\alpha-1}\Big[\|x_k-x^*\|^2-\|x_{k-1}-x^*\|^2\Big].
\end{eqnarray*}
In other words,
$$(k+\alpha-1)(h_{k+1}-h_k)-(k-1)(h_k-h_{k-1})\le 2(k+\alpha-1)\|x_k-x_{k-1}\|^2.$$
Injecting this in \eqref{E:delta_k_k-1}, we deduce that
$$\delta_{k+1} -  \delta_{k} \le 2(k+ \alpha -1)\|x_k-x_{k-1}\|^2.$$
Since the right-hand side is summable and $(\delta_k)$ is bounded, $\lim\limits_{k\to\infty}\delta_k$ exists. It follows that $\lim\limits_{k\to\infty}\| z_k -x^{*} \|$ exists. In view of Theorem \ref{T:main} and the definition \eqref{algo9c} of $z_{k}$, $\lim\limits_{k\to\infty}\|x_k-x^*\|$ exists. Since this holds for any $x^*\in S$, Opial's Lemma shows that the sequence $(x_k)$ converges weakly, as $k\to+\infty$, to a point in $S$.
\end{proof}

\subsection{Stability under additive errors} Consider the inexact version of Algorithm \eqref{algo7b} given by
\begin{equation}\label{E:FISTA_err}
\left\{
\begin{array}{rcl}
y_k&= &  x_{k} + \frac{k -1}{k  + \alpha -1} ( x_{k}  - x_{k-1}) \\
\rule{0pt}{20pt}
x_{k+1}& = &\mbox{prox}_{s \Phi} \left( y_k- s (\nabla \Psi (y_k) -g_k)\right). 
\end{array}\right.
\end{equation}
The second relation means that
$$y_k- s\nabla \Psi (y_k)\in x_{k+1}+s\Big(\partial\Phi(x_{k+1})+B(0,\varepsilon_{k+1})\Big)$$
for any $\varepsilon_{k+1}>\|g_k\|$. It turns out that it is possible to give a tolerance estimation for the sequence of errors $(g_k)$ in order to ensure that all the asymptotic properties of \eqref{algo7b} (including the $o(k^{-2})$ order of convergence) hold for \eqref{E:FISTA_err}. More precisely, we have the following:

\begin{Theorem}\label{T:errors}
Let  $\Psi: \mathcal H \to  \mathbb R \cup \lbrace + \infty  \rbrace  $ be proper, lower-semicontinuous and convex, and let $\Phi: \mathcal H \to  \mathbb R$ be convex and continuously differentiable with $L$-Lipschitz continuous gradient. Suppose that $S=\argmin (\Psi + \Phi)\neq\emptyset$, and let
$(x_k)$ be a sequence generated by algorithm \eqref{E:FISTA_err} with $\alpha > 3$ and $ 0< s < \frac{1}{L} $. If $\sum_{k=1}^\infty k\|g_k\|<+\infty$, then, the function values and the velocities satisfy $\lim\limits_{k\to\infty}k^2\Big((\Psi + \Phi)(x_k)-  \min(\Psi + \Phi)\Big) = 0$ and $\lim\limits_{k\to\infty}k\|x_{k+1}-x_{k}\|=0$, respectively. Moreover, $(x_k)$ converges weakly to a point in $S$.
\end{Theorem}

The key idea is to observe that, for each $k\ge 1$, we have 
$$\mathcal E(k) \leq \mathcal E(0) + \sum_{j=0}^{k-1} 2s\left( j + \alpha -1\right)    \left\langle g_j, z_{j+1}- x^{*}  \right\rangle$$
(with the same definitions of $z_k$ and $\mathcal E(k)$ given in \eqref{algo9c} and \eqref{algo9b}, respectively). This implies
$$\|z_{k}-x^{*}\|^2 \leq \frac{1}{\alpha -1}\mathcal E(0) + \frac{2s}{\alpha-1}\sum_{j=1}^{k}\left(j + \alpha-2\right)\|g_{j-1}\|\|z_{j}-x^{*}\|.$$
Then, we apply Lemma \cite[Lemma A.9]{ACPR} with $a_k = \| z_{k} -x^{*} \|$ to deduce that the sequence $(z_k)$ is bounded and so, the modified energy sequence $(\mathcal F(k))$, given by
$$\mathcal F (k):= \frac{2s}{\alpha -1} \left( k + \alpha -2\right)^2    (\Theta (x_k) - \Theta(x^{*}) + (\alpha -1)
 \| z_k -x^{*} \|^2  + \sum_{j=k}^{\infty} 2s\left( j + \alpha -1\right)    \left\langle g_j, z_{j+1}- x^{*}  \right\rangle,$$
is well defined and nonincreasing. The rest of the proof follows pretty much the arguments given above with $\mathcal E$ replaced by $\mathcal F$ (see also \cite[Section 5]{ACPR}).\\

Inexact FISTA-like algorithms have also been considered in \cite{SLB,VSBV}. It would be interesting to obtain similar order-of-convergence results under {\it relative error} conditions.\\

\noindent{\bf Acknowledgement.} The authors thank Patrick Redont for his valuable remarks.


\begin{thebibliography}{10}

\bibitem{AA1} {\sc F. Alvarez, H. Attouch}, {\it  An inertial proximal method for maximal monotone operators via discretization of a nonlinear oscillator with damping}, Set-Valued Analysis,  9 (2001), No. 1-2, pp.  3--11. \smallskip
\bibitem{ACPR} {\sc H. Attouch, Z. Chbani, J. Peypouquet, P. Redont},  {\it Fast convergence of inertial dynamics and algorithms with asymptotic vanishing damping}, Paper under review.
\bibitem{BC}{\sc H. Bauschke, P. Combettes}, {\it  Convex analysis and monotone operator theory in Hilbert spaces}, CMS Books in Mathematics, Springer,   (2011). \smallskip
\bibitem{BT}{\sc A. Beck, M. Teboulle},  {\it  A fast iterative shrinkage-thresholding algorithm for linear inverse problems},  SIAM J. Imaging Sci., 2  (2009),  No. 1, pp.~183--202. \smallskip
\bibitem{CD}{\sc  A. Chambolle, C. Dossal}, {\it  On the convergence of the iterates of Fista}, HAL Id: hal-01060130 https://hal.inria.fr/hal-01060130v3
Submitted on 20 Oct 2014. \smallskip
\bibitem{Chamb_Pock} {\sc A. Chambolle, T. Pock}, {\it A remark on accelerated block coordinate descent for computing the proximity operators of a sum of convex functions}, SMAI Journal of Computational Mathematics 1 (2015), pp. 29--54. \smallskip
\bibitem{ComWaj} {\sc P.L. Combettes, V.R. Wajs}, {\it Signal recovery by proximal forward-backward splitting}, Multiscale Model. Simul., 4 (2005), pp. 1168--1200.\smallskip
\bibitem{DauDefDeM} {\sc I. Daubechies, M. Defrise, C. De Mol,} {\it An iterative thresholding algorithm for linear
inverse problems with a sparsity constraint}, Comm. Pure Appl. Math., 57 (2004), pp. 1413--1457.\smallskip 
\bibitem{Gol} {\sc A.A. Goldstein}, {\it Convex programming in Hilbert space}, Bulletin of the American Mathematical Society 70 (1964) pp. 709--710.\smallskip
\bibitem{LevPol} {\sc E.S. Levitin, B.T. Polyak}, {\it Constrained minimization problems}, USSR Computational Mathematics and Mathematical Physics 6 (1966) pp. 1--50. \smallskip
\bibitem{LioMer} {\sc P.L. Lions, B. Mercier}, {\it Splitting algorithms for the sum of two nonlinear operators}, SIAM J. Numer. Anal., 16 (1979), pp. 964--979.\smallskip
\bibitem{May} {\sc R. May}, {\it Asymptotic for a second order evolution equation with convex potential and vanishing damping term}, arXiv:1509.05598. \smallskip 
\bibitem{MO}{\sc A. Moudafi, M. Oliny}, {\it  Convergence of a splitting inertial proximal method for monotone operators}, J. Comput. Appl. Math., 155  (2003), No. 2, pp.   447--454. \smallskip
\bibitem{Nest1}{\sc  Y. Nesterov}, {\it   A method of solving a convex programming problem with convergence rate
$\mathcal O(1/k^2)$}, Soviet Mathematics Doklady,  27  (1983), pp.~ 372--376. \smallskip
\bibitem{Nest2}{\sc  Y. Nesterov}, {\it  Introductory lectures on convex optimization: A basic course}, volume 87 of Applied Optimization. Kluwer Academic Publishers, Boston, MA, 2004. \smallskip
\bibitem{Nest3}{\sc  Y. Nesterov}, {\it  Smooth minimization of non-smooth functions}, Mathematical programming, 103 (2005), No. 1, pp.~127--152. \smallskip
\bibitem{Nest4}{\sc  Y. Nesterov}, {\it  Gradient methods for minimizing composite objective function}, CORE Discussion Papers, 2007. \smallskip
\bibitem{Op} {\sc Z. Opial}, {\it  Weak convergence of the sequence of successive approximations for nonexpansive mappings},  Bull. Amer. Math. Soc.,  73  (1967), pp. 591--597. \smallskip
\bibitem{PB}{\sc N. Parikh, S. Boyd}, {\it Proximal algorithms},  Foundations and trends in optimization, volume 1, (2013), pp. 123-231. \smallskip
\bibitem{Pas} {\sc G.B. Passty}, {\it Ergodic convergence to a zero of the sum of monotone operators in Hilbert space}, J. Math. Anal. Appl., 72 (1979), pp. 383--390.\smallskip 
\bibitem{Pey}{\sc J. Peypouquet}, {\it Convex optimization in normed spaces: theory, methods and examples}. Springer, 2015. \smallskip
\bibitem{LP}{\sc D.A. Lorenz, T. Pock}, {\it  An inertial forward-backward algorithm for monotone inclusions}, J. Math. Imaging Vision, pp. 1-15, 2014. (online). \smallskip
\bibitem{SLB}{\sc M. Schmidt, N. Le Roux, F. Bach}, {\it  Convergence rates of inexact proximal-gradient methods for convex optimization}, NIPS'11 - 25 th Annual Conference on Neural Information Processing Systems, Dec 2011, Grenada, Spain. (2011) HAL inria-00618152v3. \smallskip
\bibitem{VSBV}{\sc S. Villa, S. Salzo, L. Baldassarres, A. Verri}, {\it Accelerated and inexact forward-backward}, SIAM J. Optim., 23  (2013), No. 3, pp. 1607--1633. \smallskip
\bibitem{SBC}{\sc W. Su, S. Boyd, E.J. Cand\`es}, {\it A Differential equation for modeling Nesterov's accelerated gradient method: theory and insights}. Neural Information Processing Systems (NIPS) 2014. 
\end{thebibliography}
\end{document}